\documentclass{amsart}

\usepackage{amsfonts}
\usepackage{latexsym}
\usepackage{fullpage}
\usepackage{amscd}
\usepackage{amsmath}
\usepackage{amssymb}
\usepackage{graphicx}
\usepackage[mathscr]{eucal}
\usepackage{xypic}

\newtheorem{theorem}{Theorem}[section]
\newtheorem{lemma}[theorem]{Lemma}
\newtheorem{corollary}[theorem]{Corollary}
\newtheorem{proposition}[theorem]{Proposition}

\theoremstyle{definition}
\newtheorem{definition}[theorem]{Definition}
\newtheorem{example}[theorem]{Example}
\newtheorem{question}[theorem]{Question}

\theoremstyle{remark}

\numberwithin{equation}{section}

\title[On Borel $\sigma$-algebras of topologies generated by two-point selections]{On  Borel $\sigma$-algebras of  topologies generated by two-point selections}

\author{S. Garcia-Ferreira}

\address{Centro de Ciencias Matem\'aticas \\
  Universidad Nacional Aut\'onoma de M\'exico, Campus Morelia,\\
  Apartado Postal 61-3, Santa Mar\'{\i}a, \\
58089, Morelia, Michoac\'an, M\'exico}

\email{sgarcia@matmor.unam.mx}



\date{}

\thanks{ PAPIIT grant no. IN-100122 }

\begin{document} 
\maketitle

\begin{abstract} A two-point selection on a set $X$ is a function $f:[X]^2 \to X$ such that $f(F) \in F$ for every $F \in [X]^2$. It is known that every two-point selection
$f:[X]^2 \to X$ induced a topology $\tau_f$ on $X$ by using the relation: $x \leq y$ if either $f(\{x,y\}) = x$ or $x = y$, for every $x, y \in X$. We are mainly concern with the  two-point selections on the real line $\mathbb{R}$. In this paper, we study the $\sigma$-algebras of Borel, each one denoted by $\mathcal{B}_f(\mathbb{R})$, of   the  topologies $\tau_f$'s defined by a two-point selection $f$ on $\mathbb{R}$. We prove that the assumption $\mathfrak{c} = 2^{< \mathfrak{c}}$ implies the existence
 of a family $\{ f_\nu : \nu < 2^\mathfrak{c} \}$ of two-point selections on $\mathbb{R}$ such that 
$\mathcal{B}_{f_\mu}(\mathbb{R}) \neq \mathcal{B}_{f_\nu}(\mathbb{R})$ for distinct $\mu,  \nu < 2^\mathfrak{c}$.
By assuming that $\mathfrak{c} = 2^{< \mathfrak{c}}$ and $\mathfrak{c}$ is regular, we also show that there are $2^{2^\mathfrak{c}}$ many $\sigma$-algebras on 
$\mathbb{R}$ that contain $[\mathbb{R}]^{\leq \omega}$  and none of them  is the $\sigma$-algebra of Borel of $\tau_f$  for any two-point selection $f: [\mathbb{R}]^2 \to \mathbb{R}$.  Several examples are given to  illustrate some properties of these Borel $\sigma$-algebras. 
\end{abstract}

\section{Introduction}

The two-point selections have been studied by several mathematicians in the areas of Topology and Measure Theory (see for instance \cite{ag}, \cite{c}, \cite{gt}, \cite{gmn}, \cite{gmnt}, \cite{gn-1} and  \cite{ns} and the survey paper \cite{gutev}).  One important application of the two-point selections, given in \cite{gn-1}, is  the definition of topologies  which have several interesting properties (see \cite{gt}, \cite{gn-1} and \cite{gutev}).  In the article \cite{ag}, the authors introduced outer measures on $\mathbb{R}$ by mean of two-point selections on $\mathbb{R}$ and provided several examples of these new outer measures. Years later, the $\sigma$-ideals of measure zero sets of these outer measures were studied in the paper \cite{gty}. In particular, it was shown that  Martin Axiom implies that:
\begin{enumerate}
\item[$\bullet$] If  $\mathcal{I}$ is a $<\mathfrak{c}$-ideal on $\mathbb{R}$ with cofinality equal to $\mathfrak{c}$\footnote{$\mathfrak{c}$ denotes de cardinality of the real numbers},  then there is a two-point selection $f$ such that $\mathcal{I}$ is the ideal of measure zero sets with respect to the outer measure induced by $f$.

\item[$\bullet$] There exists a two-point selection $f$ such that the $\sigma$-ideal  of meager subsets of $\mathbb{R}$ is the ideal of measure zero sets with respect to the outer measure induced by $f$.

\item[$\bullet$]  $CH$ is equivalent to the existence of a two-point selection $f$ on $\mathbb{R}$ for which:
\[
\lambda_f(A)=
\begin{cases}
 0  & \text{ if $|A| \leq \omega$,}\\
+\infty & \text{ otherwise,}
\end{cases}
\]
where $\lambda_f$ is the outer measure on $\mathbb{R}$ associated to $f$.
\end{enumerate}
Also in the paper \cite{gty}, it is proved that there are $2^\mathfrak{c}$ many pairwise distinct $\sigma$-ideals on $\mathbb{R}$ which are the measure zero ideals of outer measures defined by  two-point selections on $\mathbb{R}$. On the other hand, it was shown  in \cite{hm} that the Sorgenfrey topology on $\mathbb{R}$ is induced by a two-point selection, say $f_S$, and we  know that $\tau_E$ and $\tau_{f_S}$ have different topological properties but they have the same Borel $\sigma$-algebra. After mentioning all these facts,  it is natural to address our attention to the study of the Borel $\sigma$-algebras on $\mathbb{R}$ generated by  the topologies induced by two-point selections on $\mathbb{R}$.  To explain  the results obtained in this work  we denote by $\tau_f$ the topology on $\mathbb{R}$ induced by a two-point selection $f: [\mathbb{R}]^2 \to \mathbb{R}$ and $\mathcal{B}_{f}(\mathbb{R})$ denotes the Borel $\sigma$-algebra generated by $\tau_f$. Now, we formulate  two questions which  motivated  the results of this article:

\begin{question}\label{q1} Is there a family $\{ f_\nu : \nu < 2^\mathfrak{c} \}$ of two-point selections on $\mathbb{R}$ such that 
$\mathcal{B}_{f_\mu}(\mathbb{R}) \neq \mathcal{B}_{f_\nu}(\mathbb{R})$ for distinct $\mu,  \nu < 2^\mathfrak{c}$?
\end{question}

In the third section, we answer Question \ref{q1}  under the assumption  $\mathfrak{c} = 2^{< \mathfrak{c}}$  (Theorem \ref{r1}).

\medskip

In Theorem  \ref{cocountablenoBorel}, we observe that the countable cocountable $\sigma$-algebra on $\mathbb{R}$ is not the Borel $\sigma$-algebra of any $\tau_f$, where $f \in Sel_2(\mathbb{R})$. Base on this, we suggest the following question.

\begin{question}\label{q2} Are there $2^{2^\mathfrak{c}}$ many $\sigma$-algebras on  $\mathbb{R}$ that contain $[\mathbb{R}]^{\leq \omega}$  and none of them  is the $\sigma$-algebra of Borel of $\tau_f$  for any  $f \in Sel_2(\mathbb{R})$?
\end{question}

Recall that the maximal number of $\sigma$-algebras on $\mathbb{R}$ is  $2^{2^{\mathfrak{c}}}$.  We show that  Question \ref{q2} holds in model of $ZFC$ where   
$\mathfrak{c} = 2^{< \mathfrak{c}}$ and $\mathfrak{c}$ is a regular cardinal (see Corollary \ref{c1}). 

\medskip

  Besides, in this paper, we give several examples of two-point selections on $\mathbb{R}$ whose  corresponding Borel  $\sigma$-algebras have some interesting properties. 

\section{Preliminaries}

For the facts about measure theory  that we shall only  mention  the reader is referred to the book \cite{cohn}.

\medskip

 Assume that $X$ is a set. Then $[X]^2 := \{ A \subseteq X : |A| = 2 \}$, $[X]^{<\omega} := \{ A \subseteq X :  A \ \text{is finite} \ \}$ and $[X]^{\leq \omega} := \{ A \subseteq X : |A| \leq \omega \}$.   A  function $f: [X]^2 \to X$
is called a {\it two-point selection} if $f(F) \in F$, for all $F \in [X]^2$.  The symbol $Sel_2(X)$ will denote the set of all
two-point selections defined on  $X$.  Every two-point selection $f:[X]^2 \to X$ has an {\it opposite two-point  selection} $\hat{f}:[X]^2 \to X$ defined as $\hat{f}\big(\{x, y\}\big)= y$ iff  $f\big(\{x, y\}\big)= x$, for every $x, y \in X$. For a given two-point selection $f$ on $X$, we say that a point $x\in X$ is {\it $f$-minimum} if $f\big(\{x, y\}\big)= x$ for every $y\in X \setminus \{x\}$ and a point $x \in X$ is an {\it $f$-maximum} if $f\big(\{x, y\}\big)= y$ for every $y\in X \setminus \{x\}$.

\medskip

 The Euclidian (standard) order  on the real numbers will be simply  denoted by $\leq$. The most common example of a two-point selection on the real line is the Euclidian two-point selection $f_{E}:\left[\mathbb{R}\right]^{2}\to \mathbb{R}$ given by the rule  $f_{E}\left(\{r,s\}\right)=r \quad \hbox{iff} \quad r<s$, for every $\{r, s\} \in [\mathbb{R}]^2$.  In what follows, we shall only consider  two-point selections defined on $\mathbb{R}$.

\medskip

Following E. Michael \cite{mi},  every two-point selection defines an order-like relation: For 
$f \in Sel_2( \mathbb{R})$ and $\{r, s\} \in [\mathbb{R}]^2$, we say $r <_f s$ if $f(\{r, s\}) = r$, and  for every $r, s \in \mathbb{R}$ we define $r \leq_f s$ if either $r <_s y$ or $r = s$. It is evident that the relation $\leq_f$ is reflexive, antisymmetric and linear, but, in general,
$\leq_f$ is not transitive. If $f \in Sel_2( \mathbb{R})$  and $r, s \in \mathbb{R}$, then the $f$-intervals are  $(r, s)_{f}:=\Big\{x \in \mathbb{R} \, : \, r <_f x <_{f} s\Big\}$, $(r ,s]_{f}:=\Big\{x \in \mathbb{R} \, : \, r <_{f} x <_f s \Big\}$, $(r,\rightarrow)_{f}:=\Big\{x \in \mathbb{R} \, : \, r <_f x \Big\}$, etc. For the Euclidian intervals we just write $(r ,s)$, $(r, s]$, $(r,\rightarrow)$, etc. In the notation $(r, s)$ we shall understand that $r < s$ and  observe that in the general notation $(r, s)_f$ we do not necessarily required  that $r <_f s$. As in the ordered spaces, an order-like relation defined by a two-point selection also induces a topology.
Indeed,  for an $f \in Sel_2(\mathbb{R})$ the topology on $\mathbb{R}$ generated by all open intervals $(\leftarrow,r)_f$ 
and $(r,\rightarrow)_f$, for $r \in \mathbb{R}$, will be denoted by $\tau_f$ and  $\mathcal{C}_f : = \{ C \subseteq \mathbb{R} :  \mathbb{R} \setminus C \in \tau_f \}$. 
For practical purposes, for every  $f \in Sel_2(\mathbb{R})$ the interior and closure of $A \subseteq \mathbb{R}$ in the topology $\tau_f$ will be simply denoted by
$Inf_f(A)$ and $cl_f(A)$, respectively. The Euclidian topology on $\mathbb{R}$ will be denoted by $\tau_E$ and its closed subsets by $\mathcal{C}_E$. The next general notation is going to be used frequently:

For  $f \in Sel_2(\mathbb{R})$ and $A, B \in [\mathbb{R}]^{<\omega} \setminus \{\emptyset\}$ with $A \cap B = \emptyset$, we set
$$
(A,B)_f := (\bigcap_{a \in A}(a,\rightarrow)_f) \cap (\bigcap_{b \in B}(\leftarrow,b)_f), 
$$
$$
[A,B)_f := (\bigcap_{a \in A}[a,\rightarrow)_f) \cap (\bigcap_{b \in B}(\leftarrow,b)_f),
$$
$$
(A,B]_f := (\bigcap_{a \in A}(a,\rightarrow)_f) \cap (\bigcap_{b \in B}(\leftarrow,b]_f) \ \text{and} 
$$
$$
[A,B]_f := (\bigcap_{a \in A}[a,\rightarrow)_f) \cap (\bigcap_{b \in B}(\leftarrow,b]_f). 
$$
If $f_E$ is the Euclidian selection, then we simply write $(A,B)$, $[A,B]$, $(A,B]$ and $[A,B]$ whenever $A, B \in [\mathbb{R}]^{<\omega} \setminus \{\emptyset\}$ and $A \cap B = \emptyset$. We remark that if a two-point selection $f$ does not have neither a $f$-maximal point nor a $f$-minimal point, then $\{ (A,B)_f : A, B \in [\mathbb{R}]^{<\omega} \setminus \{\emptyset\} \ \text{and} \ A \cap B = \emptyset\}$ is a base for $\tau_f$. An important  topological property of the topologies
$\tau_f$'s is stated in the next theorem.

\begin{theorem}\label{1.3}{\bf \cite{hm}}
For  for every set $X$ and for every  $f \in Sel_2(X)$, the
topology $\tau_f$ is Tychonoff (i. e., completely regular  and Hausdorff).
\end{theorem}

It the paper \cite{gt}, the authors constructed a two-point selection $f \in   Sel_2(\mathbb{P})$, where $\mathbb{P}$ denotes the irrational numbers, such that  
$(\mathbb{P},\tau_f)$ is not normal.

\medskip

Given a two-point selection $f \in Sel_2(\mathbb{R})$, the Borel  $\sigma$-algebra generated by the topology $\tau_f$ will be simply denoted by $\mathcal{B}_f(\mathbb{R})$, and we shall use the notation $\mathcal{B}(\mathbb{R})$ for the Borel $\sigma$-algebra generated by the Euclidian topology on  $\mathbb{R}$. 
 Observe that  $\mathcal{B}_f(\mathbb{R}) = \mathcal{B}_{\hat{f}}(\mathbb{R})$, recall that $\hat{f}$ denotes the opposite two-point selection of $f$,  for every $f \in Sel_2(\mathbb{R})$. 

\medskip

For every  uncountable set $X$, we let $\mathcal{C}(X) = \{ A \subseteq X : \ \text{either } \  A \ \text{ is countable} \  \text{or} \ X \setminus A \ \text{is countable}  \}$ denote the countable cocountable $\sigma$-algebra on $X$. We remark that  $\mathcal{C}(\mathbb{R} ) \subseteq \mathcal{B}_f(\mathbb{R})$ for all $f \in Sel_2(\mathbb{R})$.

\medskip

The next lemma will help us to avoid the $f$-minimal points and the $f$-maximal points of a two-point selection $f \in Sel_2(\mathbb{R})$ to study our Borel $\sigma$-algebras. 

\begin{lemma}\label{lem0}  For every $f \in   Sel_2(\mathbb{R})$ there is  $g \in   Sel_2(\mathbb{R})$ without neither a $g$-minimal point nor a $g$-maximal point such that $\mathcal{B}_f(\mathbb{R}) = \mathcal{B}_g(\mathbb{R})$.
\end{lemma}

\begin{proof} Without loss of generality suppose that $p \in \mathbb{R}$ is an $f$-minimal point and $q \in \mathbb{R}$ is an $f$-maximal point (of course, the argument also hold when only one of these two points exists). Take an infinite subset
$\{ r_n : n \in \mathbb{N} \}$ of $\mathbb{R} \setminus \{p,q\}$. For every $n \in \mathbb{N}$  we define:
\begin{enumerate} 
\item[$i)$] $r_{2n} <_g  r$  for every $r \in \mathbb{R} \setminus (\{ r_m : m \in \mathbb{N}\}$,  

\item[$ii)$]  $r <_g r_{2n+1}$ for every  $r \in \mathbb{R} \setminus (\{ r_m : m \in \mathbb{N}\}$,

\item[$iii)$]  $r_{2n+2 } <_g r_{2n}$ and

\item[$iv)$]   $r_{2n+1 }<_g r_{2n+3}$.
\end{enumerate}
For other pair of points not yet considered $g$ is defined as $f$. Observe that $g$ does not have neither a $g$-minimal point nor a $g$-maximal point. Recall that  $\{  A \subseteq \mathbb{R} : |A| = \omega \ \text{and} \  |\mathbb{R} \setminus A| = \omega  \}$ is contained in  both $\mathcal{B}_f(\mathbb{R})$ and $\mathcal{B}_g(\mathbb{R})$. Put $P := \{ r_{2n} : n \in \mathbb{N} \}$ and $Q := \{ r_{2n+1} : n \in \mathbb{N} \}$. Now, fix $A, B \in [\mathbb{R}]^{< \omega} \setminus \{\emptyset\}$ with $A \cap B = \emptyset$ and define $A_1 := A \cap Q$ and $A_2 := A \setminus A_1$ and  $B_1 := B \cap P$ and $B_2 := B \setminus B_1$. Then we have that:
\begin{enumerate}
\item $(A,B)_f \setminus \{ r_n : n \in \mathbb{N} \} = ((A\setminus A_1) \cup B_1,(B \setminus B_1) \cup A_1)_g \setminus \{ r_n : n \in \mathbb{N} \}$,

\item $[p,b)_f  \setminus \{ r_n : n \in \mathbb{N} \} = \bigcap_{n \in \mathbb{N}}(r_{2n},b)_g$ for every $b \in  \mathbb{R} \setminus (\{ r_m : m \in \mathbb{N}\} \cup \{p\})$, and

\item $(a,q]_f  \setminus \{ r_n : n \in \mathbb{N} \} = \bigcap_{n \in \mathbb{N}}(a,r_{2n+1})_g$ for every $a \in \mathbb{R} \setminus (\{ r_m : m \in \mathbb{N}\} \cup \{q\})$.
\end{enumerate}
Hence, we deduce that $\mathcal{B}_f(\mathbb{R}) =  \mathcal{B}_g(\mathbb{R})$.
\end{proof}

Based on Lemma \ref{lem0}, we shall assume that all our two-point selections do not have neither a minimal point nor a maximal point. 

\medskip

Let us recall how an outer measure on $\mathbb{R}$ is constructed in \cite{ag} by using a two-point selection:

\medskip

If $f: [\mathbb{R}]^2 \to \mathbb{R}$ is a two-point selection and $A \subseteq \mathbb{R}$, then we define
$$
\lambda_{f}(A):=\inf\Big\{\sum_{n \in \mathbb{N}} |s_{n} - r_{n}| \, : \, A \subseteq \bigcup_{n \in \mathbb{N}}(r_{n}, s_{n}]_{f} \Big\},
$$
if $A$ can be cover for a countable family of semi-open $f$-intervals, and   $\lambda_{f}(A)=+\infty$ otherwise. Observe that the Lebesgue outer measure $\lambda$ coincides with  the outer measure $\lambda_{f_E}$. The $\sigma$-algebra of $\lambda_f$-measurable subsets will be denoted by  $\mathcal{M}_f$ and $\mathcal{N}_f$ stands for the $\sigma$-ideal of $\lambda_f$-null subsets of $\mathbb{R}$,  for each $f \in Sel_2(\mathbb{R})$.  We remark that   $\mathcal{N}_f \subseteq \mathcal{M}_f$,  for all $f \in Sel_2(\mathbb{R})$. 

\medskip

To carry out our main goals we shall need the following tool and facts from Set Theory.

\begin{definition} Let $\kappa$ be an infinite cardinal number. An infinite  family $\mathcal{A} \subseteq [\kappa]^\kappa$ is called $\kappa$-{\it almost disjoint}
if $|A \cap B| < \kappa$ for distinct $A, B \in \mathcal{A}$.
\end{definition}

For an infinite cardinal number $\kappa$, we set $2^{< \kappa} := \sup \{ 2^\lambda : \lambda < \kappa \ \text{and} \ \lambda \ \text{is a cardinal number}  \}$.
The following lemma is Corollary $12.3(c)$ of the book \cite{cn74}.

\begin{lemma}\label{cn} If $\kappa$ is an infinite cardinal number such that $\kappa = 2^{< \kappa}$, then there is a $\kappa$-{\it almost disjoint}
family $\mathcal{A}$ with $|\mathcal{A}| = 2^\kappa$. 
\end{lemma}

The kind of almost disjoint families  that we will need, under certain set-theoretical assumption, are the next.
 
\begin{corollary}\label{cn-coro} If $\mathfrak{c} = 2^{< \mathfrak{c}}$, then there is a $\mathfrak{c}$-{\it almost disjoint}
family $\mathcal{A}$ with $|\mathcal{A}| = 2^\mathfrak{c}$. 
\end{corollary}

\section{Borel $\sigma$-algebras}

We start this section with a very illustrative trivial  example.

\begin{example}\label{discrete} For every infinite set $X$ there is $f \in Sel_2(X)$ such that $\tau_f = \mathcal{P}(X)$ is the discrete topology. 
\end{example}

\begin{proof}  Let $\{ P_\nu : \nu  < \mathfrak{c} \}$ be a partition of $X$ in infinite countable subsets  and put $P_\nu = \{ x^n_\nu : n \in \}$
for each $\nu < |X|$. Now, define  $f \in Sel_2(X)$ as follows:
$$
f\left(\{x^m_\mu,x^n_\nu\}\right):=  \left\{ \begin{array}{rcl} x^m_\mu&\mbox{if} \quad \mu = \nu  \quad \mbox{and} \quad m < n \\\\
 x^m_\mu&\mbox{if} \quad \mu < \nu\\\\
\end{array}\right.
$$
for each $m, n \in \mathbb{N}$ and for each $\mu, \nu < \mathfrak{c}$. It is not hard to see that $\tau_f$ is the discrete topology.
\end{proof}

Thus, in the case of the real line, there is $f \in Sel_2(\mathbb{R})$ such that  $\mathcal{B}_f(\mathbb{R}) = \mathcal{P}(\mathbb{R})$. We will see next that there are many $\sigma$-algebras of the form $\mathcal{B}_f(\mathbb{R})$ pairwise distinct. To have this done we need the following notions and preliminary results:

\begin{lemma}\label{tensor} Let $X$ be a set and let $\{Y, Z\}$ be a partition of $X$ in two nonempty sets. If $\mathcal{A}$ and $\mathcal{B}$ are $\sigma$-algebras on $Y$ and $Z$, respectively, then $\mathcal{A} \oplus \mathcal{B} : = \{ A \cup B : A \in \mathcal{A} \ \text{and} \ B \in \mathcal{B} \}$ is a  $\sigma$-algebra on $X$.
\end{lemma}

\begin{definition} Let $Z$ be a set and let $\{X, Y\}$ be a partition of $X$ in two sets. If $f \in Sel_2(X)$ and $g \in Sel_2(Y)$, respectively, then the function
$$
h\left(\{x,y\}\right):=  \left\{ \begin{array}{rcl} x \ \ \ \ \ &\mbox{if} \quad x \in X \quad \mbox{and} \quad y \in Y \\\\
 f(\{x,y\})&\mbox{if} \quad x, y \in X\\\\
g(\{x,y\})&\mbox{if} \quad x, y \in Y,\\\\
\end{array}\right.
$$
for every $\{x,y\} \in [Z]^2$, is a two-point selection which is denoted by $f \oplus g$. 
\end{definition}

\begin{theorem}\label{oplus} Let $\{X, Y\}$ be a partition of $\mathbb{R}$ in two infinite subsets. If $f \in Sel_2(X)$ and $g \in Sel_2(Y)$,  then 
$$
\mathcal{B}_{f\oplus g}(\mathbb{R}) = \mathcal{B}_f(X) \oplus \mathcal{B}_g(Y).
$$
\end{theorem}

\begin{proof} Put $h := f \oplus g$. For every $r \in \mathbb{R}$ we know that:
\begin{enumerate}
\item[$\bullet$] $(r,\rightarrow)_h = (r,\rightarrow)_f \cup Y$ if $r \in X$;

\item[$\bullet$]  $(r,\rightarrow)_h = (r,\rightarrow)_g$ if $r \in Y$;

\item[$\bullet$]  $(\leftarrow,r)_h = (\leftarrow,r)_g \cup X$ if $r \in Y$; and

\item[$\bullet$]  $(\leftarrow,r)_h = (\leftarrow,r)_f$ if $r \in X$.
\end{enumerate}
Hence, we deduce directly that $\mathcal{B}_{h}(\mathbb{R}) \subseteq \mathcal{B}_f(X) \oplus \mathcal{B}_g(Y)$. 
Also, we have that 
$$
(x,y)_h = (x,y)_f 
$$
for every $x, y \in X$. This implies that $\tau_f \subseteq  \tau_h$ and so $\mathcal{B}_{f}(\mathbb{R}) \subseteq \mathcal{B}_{h}(\mathbb{R})$.
In a similar way, we prove that $\mathcal{B}_{g}(\mathbb{R}) \subseteq \mathcal{B}_{h}(\mathbb{R})$. So $\mathcal{B}_f(X) \oplus \mathcal{B}_g(Y) \subseteq \mathcal{B}_{h}(\mathbb{R})$. Therefore, $\mathcal{B}_{f\oplus g}(\mathbb{R}) = \mathcal{B}_f(X) \oplus \mathcal{B}_g(Y)$.
\end{proof}

We let  $[\mathbb{R}]^\mathfrak{c} := \{ X \subseteq \mathbb{R} : |X|  = \mathfrak{c}\}$.
For every $X \in [\mathbb{R}]^\mathfrak{c}$, we let $f_X \in Sel_2(X)$ be the two-point selection defined in Example \ref{discrete} such that $\tau_{f_X}$ is the discrete topology on $X$, and let $f_E^X := f_E \upharpoonright_{[X]^2}$; that is, $f_E^X$ is the Euclidian two-point selection on $X$ as a subspace of $\mathbb{R}$.

\begin{lemma} Let $X, Y \in [\mathbb{R}]^\mathfrak{c}$ be such that $|X \setminus Y| = \mathfrak{c}$. Then, 
$$
\mathcal{B}_{f_ X \oplus f_E^{\mathbb{R} \setminus X}}(\mathbb{R}) \neq \mathcal{B}_{f_ Y \oplus f_E^{\mathbb{R}\setminus Y}}(\mathbb{R}).
$$ 
\end{lemma}

\begin{proof} It suffices to show that $\mathcal{P}(X \setminus Y) \subseteq \mathcal{B}_{f_ X \oplus f_E^{\mathbb{R}\setminus X}}(\mathbb{R})$ and
$\mathcal{P}(X \setminus Y) \not \subseteq \mathcal{B}_{f_ Y \oplus f_E^{\mathbb{R}\setminus Y}}(\mathbb{R})$. Indeed, observe from Lemma \ref{oplus} that 
$\mathcal{P}(X \setminus Y) \subseteq \mathcal{P}(X) = \mathcal{B}_{f_X}(\mathbb{R}) \subseteq \mathcal{B}_{f_ X \oplus f_E^{\mathbb{R} \setminus X}}(\mathbb{R})$.
Since $|\mathcal{B}_{f_E^{\mathbb{R} \setminus Y}}(\mathbb{R})| = \mathfrak{c} < |\mathcal{P}(X \setminus Y)| = 2^{\mathfrak{c}}$, we have that  $\mathcal{P}(X \setminus Y) \not \subseteq \mathcal{B}_{f_E^{\mathbb{R} \setminus Y}}(\mathbb{R})$ and it is clear that 
$\mathcal{P}(X \setminus Y) \not \subseteq \mathcal{B}_{f_Y}(\mathbb{R})$. According to Lemma \ref{oplus}, we obtain that 
$\mathcal{P}(X \setminus Y) \not \subseteq \mathcal{B}_{f_ Y \oplus f_E^{\mathbb{R}\setminus Y}}(\mathbb{R})$.
\end{proof}

Now, we solve positively Question \ref{q1} under a set-theoretic assumption: 

\begin{theorem}\label{r1} If $\mathfrak{c} = 2^{< \mathfrak{c}}$, then   there  is a family $\{ f_\nu : \nu < 2^\mathfrak{c} \}$ of two point selections on $\mathbb{R}$ such that $\mathcal{B}_{f_\mu}(\mathbb{R}) \neq \mathcal{B}_{f_\nu}(\mathbb{R})$ for distinct $\mu,  \nu < 2^\mathfrak{c}$.
\end{theorem}

\begin{proof} Assume that $\mathfrak{c} = 2^{< \mathfrak{c}}$. In virtue of Corollary \ref{cn-coro}, there is a $\mathfrak{c}$-{\it almost disjoint}
family $\mathcal{A}$ with $|\mathcal{A}| = 2^\mathfrak{c}$.  For each $\nu < \mathfrak{c}$, we define $f_\nu := f_ {A_\nu} \oplus f_E^{\mathbb{R} \setminus A_\nu}$.
 By Theorem \ref{oplus}, we obtain that $\mathcal{B}_{f_\mu}(\mathbb{R}) \neq \mathcal{B}_{f_\nu}(\mathbb{R})$ whenever $\mu < \nu < \mathfrak{c}$.
\end{proof}

Our  next task is to prove that  the $\sigma$-algebra $\mathcal{C}(\mathbb{R})$ is not the $\sigma$-algebra of Borel of a topology $\tau_f$ for any $f \in Sel_2(\mathbb{R})$. First of all, we shall prove a general lemma.

\begin{lemma}\label{cocountable} Let $(X,\tau)$ be space where $X$ is an uncountable set.  If $\tau \subseteq \mathcal{C}(X)$, then 
we have that  every discrete subset of $X$ is countable.
\end{lemma}

\begin{proof} Suppose that $D := \{ x_\alpha : \alpha < \omega_1\}$ is a discrete subset of $X$. Then for every $\alpha < \omega_1$ choose an open subset $V_\alpha$
so that $D \cap V_\alpha = \{x_\alpha\}$.  Hence, we have that the open set $\bigcup \{ V_\alpha : \alpha <  \omega_1 \ \text{and } \ \alpha \ \text{is even} \}$ cannot be
in $\mathcal{C}(X)$.
\end{proof}

\begin{theorem}\label{cocountablenoBorel} The countable cocountable $\sigma$-algebra $\mathcal{C}(\mathbb{R})$ is different from any $\sigma$-algebra of Borel $\mathcal{B}_f(\mathbb{R})$ with $f \in Sel_2(\mathbb{R})$.
\end{theorem}

\begin{proof}  Fix $f \in Sel_2(\mathbb{R})$. Suppose that either $|(\leftarrow,r)_f| \leq \omega$ or $|(r,\rightarrow)_f| \leq \omega$ for every $r \in \mathbb{R}$. Define
$L := \{ r \in \mathbb{R} : |(\leftarrow,r)_f| \leq \omega \}$ and $R := \{ r \in \mathbb{R} : |(r,\rightarrow)_f| \leq \omega \}$. Without loss of generality, we may assume that
$|L| = \mathfrak{c}$. Now, fix any $r_0 \in L$ and assume that we have defined $r_\alpha \in L$ for each $\alpha < \gamma < \omega_1$ so that 
$\alpha < \beta < \omega_1$ iff $r_\alpha <_f r_\beta$, for $\alpha, \beta < \gamma$.  Since $|\bigcup_{\alpha < \gamma}(\leftarrow,r_\alpha)_f | \leq \omega$, we can choose $r_\gamma \in L \setminus \big(\bigcup_{\alpha < \gamma}(\leftarrow,r_\alpha]\big)$. So $r_\alpha <_f r_\gamma$ for every $\alpha < \gamma$. Now, consider the set $Y = \{ r_\alpha : \alpha < \omega_1\}$. Since the topology $\tau_f\upharpoonleft_{Y}$ on $Y$ inherited from $\tau_f$ is homeomorphic to $\omega_1$ with the order topology, we have that $Y$ has a discrete set of size $\omega_1$ which is also discrete in $\mathbb{R}$ under the topology  $\tau_f$. but this contradicts Lemma \ref{cocountable}. 
\end{proof}

To continue with our plan  we need  to recall some facts about the $\sigma$-algebra generated by a family of subsets of a given set and we also need to prove a preliminary lemma.

\medskip

Given $\emptyset \neq \mathcal{A} \subseteq \mathcal{P}(X)$, the $\sigma$-algebra on $X$ generated by $\mathcal{A}$ will be denoted by $\langle \mathcal{A} \rangle$.
One  of the inductive constructions of $\langle \mathcal{A} \rangle$ is the following:
\begin{enumerate}
\item[$\bullet$] First we define $\mathcal{A}_0 :=  \mathcal{A} \cup \{ X \setminus A : A \in \mathcal{A}\}$.

\item[$\bullet$] $\mathcal{A}_1 := \{ \bigcap_{n \in \mathbb{N}}A_n :  \forall n \in \mathbb{N}(A_n \in \mathcal{A}_0) \}$.

\item[$\bullet$] $\mathcal{A}_2 := \{ \bigcup_{n \in \mathbb{N}}A_n :  \forall n \in \mathbb{N}(A_n \in \mathcal{A}_1) \}$.

$$\vdots \ \ \ \ \ \ \ \ \ \ \ \ \vdots \ \ \ \ \ \ \ \ \ \ \ \ \vdots$$

\item[$\bullet$] $\mathcal{A}_\alpha := \{ \bigcup_{n \in \mathbb{N}}A_n :  \forall n \in \mathbb{N}(A_n \in \bigcup_{\beta < \alpha}\mathcal{A}_\beta) \}$ if $\alpha < \omega_1$ and $\alpha$ is even.

\item[$\bullet$] $\mathcal{A}_\alpha := \{ \bigcap_{n \in \mathbb{N}}A_n :  \forall n \in \mathbb{N}(A_n \in \bigcup_{\beta < \alpha}\mathcal{A}_\beta) \}$ if $\alpha < \omega_1$ and $\alpha$ is odd.
\end{enumerate}
Finally, we have that $\langle \mathcal{A} \rangle = \bigcup_{\alpha < \omega_1}\mathcal{A}_\alpha$.

\begin{lemma}\label{sigma} Let $X$ be an infinite set and $\emptyset \neq \mathcal{A} \subseteq \mathcal{P}(X)$. Then for every $B \in \langle \mathcal{A} \rangle$ there is  $\{ A_n : n \in \mathbb{N}\} \subseteq \mathcal{A}$ such that either
\begin{enumerate}
\item $B \subseteq \bigcup_{n \in \mathbb{N}}A_n$ or 

\item $X \setminus B \subseteq \bigcup_{n \in \mathbb{N}}A_n$.
\end{enumerate}
\end{lemma}

\begin{proof} Fix $B \in \langle \mathcal{A} \rangle$ and put $\langle \mathcal{A} \rangle = \bigcup_{\alpha < \omega_1}\mathcal{A}_\alpha$. It is clear that if
$B \in \mathcal{A}_0 =  \mathcal{A} \cup \{ X \setminus A : A \in \mathcal{A}\}$, then  either $(1)$ or $(2)$ holds.  Fix $\alpha < \omega_1$ and
 assume that if $\beta < \alpha$ , then either  $(1)$ or $(2)$ holds for each element of $\mathcal{A}_\beta$. Suppose that  $B \in \mathcal{A}_\alpha$. 
 Then for every $n \in  \mathbb{N}$ there are $\beta_n < \alpha$,  $B_n \in  \mathcal{A}_{\beta_n}$ and $\{ A^n_m : m \in 
 \mathbb{N}\} \subseteq \mathcal{A}$ such that $B$ is determined by  the sets  $\{ B_n : n \in \mathbb{N}\}$  and  either
\begin{enumerate}
\item $B_n \subseteq \bigcup_{m \in \mathbb{N}}A_m^n$ or 

\item $X \setminus B_n \subseteq \bigcup_{m \in \mathbb{N}}A_m^n$.
\end{enumerate}
Now, let $I := \{ n \in \mathbb{N} : B_n \subseteq \bigcup_{m \in \mathbb{N}}A_m^n \}$ and $J := \{ n \in \mathbb{N} : X \setminus B_n \subseteq \bigcup_{m \in \mathbb{N}}A_m^n\}$. We consider the following two cases:

\medskip

Case I. The ordinal $\alpha$ is even.  In this case,  we know that $B = \bigcup_{n \in \mathbb{N}}B_n$. If $J = \emptyset$, then we are done. Suppose that $J \neq \emptyset$. Then,  $X \setminus B \subseteq X \setminus B_n \subseteq \bigcup_{n \in \mathbb{N}}A_n$ for every $n \in J$. Thus, condition $(2)$ holds. 

\medskip

Case II. The ordinal $\alpha$ is odd. Then we have that $B = \bigcap_{n \in \mathbb{N}}B_n$. If $I \neq \emptyset$, then $(1)$ holds. Suppose that $J = \mathbb{N}$. Then, it follows that $X \setminus B = X \setminus \big(\bigcap_{n \in \mathbb{N}}B_n \big) = \bigcup_{n \in \mathbb{N}}X \setminus B_n \subseteq \bigcup_{n \in \mathbb{N}} \big(\bigcup_{m \in \mathbb{N}}A_m^n \big)$. So, condition $(2)$ holds.
\end{proof}

To prove that Question \ref{q2}  holds, under certain set-theoretic assumptions, it suffices to show the existence $2^{2^{\mathfrak{c}}}$ many $\sigma$-algebra that contain $[\mathbb{R}]^{\leq \omega}$. 

\begin{theorem}\label{muchos} Assume that $\mathfrak{c} = 2^{< \mathfrak{c}}$ and $\mathfrak{c}$ is regular. Then there are $2^{2^{\mathfrak{c}}}$ many $\sigma$-algebras on $\mathbb{R}$ that contain $[\mathbb{R}]^{\leq \omega}$ and they are pairwise distinct. 
\end{theorem} 

\begin{proof} Assume that  $\mathfrak{c} = 2^{< \mathfrak{c}}$ and $\mathfrak{c}$ is a regular cardinal. By Lemma \ref{cn},  we can fix  a $\mathfrak{c}$-almost disjoint
family $\mathcal{A}$ on $\mathbb{R}$ with $|\mathcal{A}| = 2^\mathfrak{c}$.  Enumerate $\mathcal{A}$ as $\{ A_\nu : \nu <  2^\mathfrak{c}\}$  and consider the set
$$
\mathcal{D} : = \{ D \subseteq 2^\mathfrak{c} : |2^\mathfrak{c} \setminus D| = |D| = 2^\mathfrak{c}\}.
$$  
Notice that $|\mathcal{D} | =  2^{2^{\mathfrak{c}}}$. For every $D \in \mathcal{D}$, we define $\mathcal{S}_D := \langle \{ A_\nu : \nu \in D\} \cup [\mathbb{R}]^{\leq \omega} \rangle$. Observe that $\{ A_\nu : \nu \in D \} \subseteq \mathcal{S}_D$ for every $D \in \mathcal{D}$. Now, fix  $E, D \in \mathcal{D}$ and suppose that there is $\mu \in E \setminus D$.  According to Lema \ref{sigma}, there are  two disjoin sets $I$ and $J$ of $ \mathbb{N}$, $\{ A_{\nu_n} : n \in I \} \subseteq \{ A_\nu : \nu \in D\}$ and  $\{ F_n : n \in J \} \subseteq [\mathbb{R}]^{\leq \omega}$ such that either
\begin{enumerate}
\item $A_\mu \subseteq \big(\bigcup_{n \in I }A_{\nu_n} \big) \cup \big( \bigcup_{n \in J }F_n \big)$ or

\item $\mathbb{R} \setminus A_\mu \subseteq \big(\bigcup_{n \in \mathbb{N}}A_{\nu_n}\big) \cup \big( \bigcup_{n \in J }F_n \big)$.
\end{enumerate}
Since $|A_\mu \cap A_\nu| < \mathfrak{c}$ for every $\nu \in D$, clause $(1)$ is impossible. Thus, we must have that 
$X \setminus A_\mu \subseteq \big(\bigcup_{n \in \mathbb{N}}A_{\nu_n}\big) \cup \big( \bigcup_{n \in J }F_n \big)$. Hence,  
 $\big(\bigcap_{n \in \mathbb{N}}\mathbb{R} \setminus A_{\nu_n} \big) \cap \big( \bigcap_{n \in J } \mathbb{R} \setminus F_n \big) \subseteq A_\mu$. Now, choose any
 $\gamma \in 2^\mathfrak{c} \setminus \big( \{ \nu_n : n \in \mathbb{N} \} \cup \{\mu\}\big)$. We know that $|A_\gamma \cap \big( \bigcup_{n \in \mathbb{N}}  A_{\nu_n}\big)| < \mathfrak{c}$ and hence $|A_\gamma \cap \big(\bigcap_{n \in \mathbb{N}}\mathbb{R} \setminus A_{\nu_n} \big) \cap \big( \bigcap_{n \in J } \mathbb{R} \setminus F_n \big)| = \mathfrak{c}$ which is also impossible. Thus, we obtain that  $A_\mu \in \mathcal{S}_E  \setminus \mathcal{S}_D$.  Therefore, $\{ \mathcal{S}_D : D
 \in \mathcal{D}   \}$ is a family of  $\sigma$-algebras on $\mathbb{R}$ that contain $[\mathbb{R}]^{\leq \omega}$ and they are pairwise distinct. 
\end{proof}

\begin{corollary}\label{c1} Assume that $\mathfrak{c} = 2^{< \mathfrak{c}}$ and $\mathfrak{c}$ is regular. Then there are $2^{2^\mathfrak{c}}$ many $\sigma$-algebras on  $\mathbb{R}$ that contain $[\mathbb{R}]^{\leq \omega}$  and none of them  is the $\sigma$-algebra of Borel of a topology $\tau_f$  for any $f \in Sel_2(\mathbb{R})$. 
\end{corollary}

We know that the one-point subsets of $\mathbb{R}$  are closed, non-open and $G_\delta$-sets in the Euclidian topology.  In Example \ref{discrete} were $\tau_f$ is the discrete topology the one-point sets are  trivially open. In the next results, we show how we can define  two-point selections to obtain a predetermined  properties of some subsets of $\mathbb{R}$.

\begin{theorem}\label{teo1} If $U \subseteq \mathbb{R}$ satisfies that $|\mathbb{R} \setminus U| = \mathfrak{c}$, then there is $f \in Sel_2(\mathbb{R})$
such that $Int_{f}(U) = \emptyset$. In particular, if $U \neq \emptyset$, then $U \notin \tau_f$. 
\end{theorem}

\begin{proof}  Enumerate $ \{ (A,B) : A, B \in [\mathbb{R}]^{<\omega} \setminus \{\emptyset\} \ \text{and} \ A \cap B = \emptyset\}$ as
$\{ (A_\nu,B_\nu) : \nu < \mathfrak{c} \}$. Now, consider a subset $\{ c_\nu : \nu < \mathfrak{c} \}$ of $\mathbb{R}$ such that 
$$
c_\nu \notin \big((\bigcup_{\mu \leq \nu}A_\mu) \cup (\bigcup_{\mu \leq \nu}B_\mu)\big) \cup U,
$$
for every $\nu < \mathbb{c}$. We are ready to define $f \in Sel_2(\mathbb{R})$ as follows: 

\smallskip

For every $\nu < \mathfrak{c}$ we set $a <_f c_\nu <_f b$ for every  $a \in A_\nu$ and $b \in B_\nu$. And $f$ is defined as the Euclidian order on the 
remaining pairs of real numbers, and it is also possible to define $f$ to avoid an $f$-maximal point  and an $f$-minimal point.

\smallskip

Observe that $c_\nu \in (A_\nu,B_\nu)$ for all $\nu < \mathfrak{c}$. If $r \in Int_{f}(U)$, then there is $\mu < \mathfrak{c}$ such that $r \in (A_\mu,B_\mu) \subseteq U$, but this is a contradiction since
$c_\mu \notin U$.
\end{proof}

The following result is a direct consequence of Theorems \ref{teo1} assuming the Continuous Hypothesis.

\begin{corollary}\label{coro1} {\bf [CH]}. If $\emptyset \neq U \subseteq \mathbb{R}$ and  $\mathbb{R} \setminus U$ is uncountable, then there is $f \in Sel_2(\mathbb{R})$
such that   $U \notin \tau_f$. 
\end{corollary}

We do not know whether or not it is possible to  remove the Continuous Hypothesis from the previous corollary.

\begin{theorem}\label{teo2} If $C \subseteq \mathbb{R}$ is infinite and $C \neq \mathbb{R}$, then there is $f \in Sel_2(\mathbb{R})$
such that   $C \notin \mathcal{C}_f$. 
\end{theorem}

\begin{proof} Enumerate a countable infinite subset $\{ c_n : n \in \mathbb{N}\}$ of $C$ and fix $r \in \mathbb{R} \setminus C$.  Now each $n \in \mathbb{N}$  choose  $\delta_n > 0$ so that $r - \delta_n \neq c_n \neq r + \delta_n$ and $\delta_n \to 0$. Then we proceed to define 
$f \in Sel_2(\mathbb{R})$ so that   $r - \delta_n <_f c_n <_f r + \delta_n$ for every  $n \in \mathbb{N}$, and $f$ is defined as the Euclidian order on the 
remaining pairs of real numbers. Notice that  $f$ does not have neither  an $f$-maximal point  nor an $f$-minimal point. Then, we have that 
$r \in cl_{\tau_f}(C) \setminus C$ and so $C \notin \mathcal{C}_f$.
\end{proof}

\begin{theorem}\label{teo3} If $C \subseteq \mathbb{R}$ satisfies that $\omega \leq |C|$ and $|\mathbb{R} \setminus C| = \mathfrak{c}$, then there is $f \in Sel_2(\mathbb{R})$ such that   $C \notin \tau_f \cup \mathcal{C}_f$. 
\end{theorem}

\begin{proof} To obtain the require two-point selection we combine the ideas used  in the proofs of Theorems \ref{teo1} and \ref{teo2}. 
\end{proof}

Next, we give an example of a two-point selection $f \in Sel_2(\mathbb{R})$ for which the singleton subsets of $\mathbb{R}$ are not  $G_\delta$-sets in $\tau_f$.

\begin{example}\label{ex1} There is  $f \in Sel_2(\mathbb{R})$ such that $\{r\}$ is not a $G_\delta$-set in $\tau_f$ for every $r \in \mathbb{R}$.
\end{example}

\begin{proof} Split $\mathbb{R} = X \cup Y$ where $X \cap Y = \emptyset$ and $|X| = |Y| = \mathfrak{c}$.  Consider the set  $\mathcal{A} =  \{ (A,B) : A, B \in [\mathbb{R}]^{< \omega} \setminus \{\emptyset\} \ \text{and} \ A \cap B = \emptyset\}$. Enumerate all sequences of $\mathcal{A}$ with some special properties stated in the enumeration:
$$
\{ \{ (A^n_\nu,B^n_\nu) : n \in \mathbb{N} \} : \nu < \mathfrak{c} \  \text{and} \ \big(\bigcup_{n \in \mathbb{N}}A_\nu^n \big) \cap \big(\bigcup_{n \in \mathbb{N}}B_\nu^n \big)  = \emptyset \}.
$$ 
The reader may figurate that this enumeration will help us to deal with the all possible $G_\delta$-subsets of the future topology $\tau_f$. Now, inductively define two subset one $\{ r_\nu : \nu < \mathfrak{c} \}$ inside of $X$ and the other  $\{ s_\nu : \nu < \mathfrak{c} \}$ inside of $Y$ so that 
$$
s_\nu \notin \big(  \bigcup_{\mu \leq \nu}\big((\bigcup_{n \in \mathbb{N}}A_\mu^n) \cup (\bigcup_{n \in \mathbb{N}}B_\mu^n)\big) \big)  \cup \{s_\mu : \mu < \nu \} \cup \{r_\mu : \mu < \nu \}
$$
and
$$
r_\nu \notin \big(  \bigcup_{\mu \leq \nu}\big((\bigcup_{n \in \mathbb{N}}A_\mu^n) \cup (\bigcup_{n \in \mathbb{N}}B_\mu^n)\big) \big)  \cup \{s_\mu : \mu \leq \nu \} \cup \{r_\mu : \mu < \nu \}
$$
for every $\nu < \mathfrak{c}$. We inductively define $f \in Sel_2(\mathbb{R})$ so that  
$a <_f r_\nu <_f b$ and $a <_f s_\nu <_f b$ for every  $a \in \bigcup_{n \in \mathbb{N}}A_\nu^n$ and for every $b \in \bigcup_{n \in \mathbb{N}}A_\nu^n$, for each $\nu < \mathfrak{c}$. After this, we proceed to define  $f$ as the Euclidian order on the remaining pairs of real numbers and avoiding a minimal  point and a maximal point. 

\smallskip

Assume that  $G$ is a nonempty $G_\delta$-set in the topology $\tau_f$.  Now suppose that $\emptyset \neq  \bigcap_{n \in \mathbb{N}}(A^n,B^n)_f \subseteq G$ where  $A^n, B^n \in [\mathbb{R}]^{<\omega} \setminus \{\emptyset\}$, for every $n \in \mathbb{N}$, and
$\big(\bigcup_{n \in \mathbb{N}}A^n \big) \cap \big(\bigcup_{n \in \mathbb{N}}B^n \big)  = \emptyset $. Choose $\mu < \mathfrak{c}$ so that 
$A^n = A^n_\mu$ and $B^n = A^n_\mu$ for all $n \in \mathbb{N}$. Since $r_\mu, s_\mu \in \bigcap_{n \in \mathbb{N}}(A_\mu^n,B_\mu^n)_f$, then we have that
$G$ contains at least two points.
\end{proof}

By using the idea of the constructions of the previous example and Theorem \ref{teo3} we may establish the following result. 

\begin{theorem}\label{teo4} If $C \subseteq \mathbb{R}$ satisfies that $\omega \leq |C|$ and $|\mathbb{R} \setminus C| = \mathfrak{c}$, then there is $f \in Sel_2(\mathbb{R})$ such that   $C \notin \mathcal{C}_f$ and $C$ is not a $G_\delta$-set in the topology $\tau_f$. 
\end{theorem}

It is shown in \cite[Ex. 3.10]{ag} the existence of a two-point selection $f \in Sel_f(\mathbb{R})$ for which $(0,1)_f \notin \mathcal{M}_f$; in this case, we have that  
$\mathcal{B}_f(\mathbb{R}) \not \subseteq \mathcal{M}_f$. On the other hand,  it is well-known that there are  Lebesgue measurable subsets of $\mathbb{R}$ which are not inside of  
$\mathcal{B}(\mathbb{R})$.  In the next theorem, we will see that for certain subsets $C'$s of $\mathbb{R}$ we can find $f \in Sel_2(\mathbb{R})$ so that $C \in \mathcal{M}_f \setminus  \mathcal{B}_{f}(\mathbb{R})$. To have this done we shall need the next lemmas.

\medskip

Given a bijection $\gamma: \mathbb{R} \to \mathbb{R}$, we denote by $f_\gamma$ the two-point selection  on $\mathbb{R}$ defined by $r <_{f_\gamma} s$ if 
$\gamma(r) < \gamma(s)$ for each $\{r, s\} \in [\mathbb{R}]^2$ (remember that $\leq$ is the standard order on $\mathbb{R}$).  We remark that 
$f_\gamma$ does not have neither a minimal point nor maximal point. 

\begin{lemma}\label{lemaA}  If $\gamma: \mathbb{R} \to \mathbb{R}$ is a bijection, then
\begin{enumerate}  
\item [$\bullet$] $(A,B)_{f_\gamma} = \gamma^{-1}\big((\gamma(A),\gamma(B))\big)$,

\item [$\bullet$] $[A,B)_{f_\gamma} = \gamma^{-1}\big([\gamma(A),\gamma(B))\big)$,

\item [$\bullet$] $(A,B]_{f_\gamma} = \gamma^{-1}\big((\gamma(A),\gamma(B)]\big)$ and

\item [$\bullet$] $[A,B]_{f_\gamma} = \gamma^{-1}\big([\gamma(A),\gamma(B)]\big)$,
\end{enumerate}
for every  $A, B \in [\mathbb{R}]^{< \omega} \setminus \{\emptyset\}$ with $A \cap B = \emptyset$.
\end{lemma}

\begin{proof} We only prove the  first equality. Fix $A, B \in [\mathbb{R}]^{< \omega} \setminus \{\emptyset\}$ with $A \cap B = \emptyset$. Then,
$$
x \in (A,B)_{f_\gamma} \Leftrightarrow a <_{f_\gamma} x <_{f_\gamma}  b \ \text{ for all} \ a \in A \ \text{ and} \ b \in B \Leftrightarrow
$$
$$  
\gamma(a) < \gamma(x) < \gamma(b) \ \text{ for all} \ a \in A \ \text{ and} \ b \in B \Leftrightarrow x \in \gamma^{-1}\big((\gamma(A),\gamma(B))\big).
$$
\end{proof}

\begin{lemma}\label{lemaB}  If $\gamma: \mathbb{R} \to \mathbb{R}$ is a bijection, then $\mathcal{B}_{f_\gamma}(\mathbb{R}) = \gamma^{-1}\big(\mathcal{B}(\mathbb{R})\big)$.
\end{lemma}

\begin{proof} First notice that $\gamma^{-1}\big(\mathcal{B}(\mathbb{R})\big)$ is a $\sigma$-algebra. It follows from Lemma \ref{lemaA} that $\gamma^{-1}(\tau_E) = \tau_{f_\gamma}$ and since $\gamma$ is a bijection, we also have that 
$\gamma^{-1}(\mathcal{C}_E) = \mathcal{C}_{f_\gamma}$. Hence, we deduce that  $\mathcal{B}_{f_\gamma}(\mathbb{R}) = \gamma^{-1}\big(\mathcal{B}(\mathbb{R})\big)$.
\end{proof}

\begin{lemma}\label{lemaC}  If $C \subseteq  \mathbb{R}$ satisfies that  $|\mathbb{R}\setminus C| = \mathfrak{c}$, then
there are $\{ a_n : n \in \mathbb{N} \} \cup \{ b_n : n \in \mathbb{N} \} \subseteq \mathbb{R}\setminus C$ such that 
$b_n - a_n \longrightarrow 0$ and $0 < b_n - a_n $ for every $n \in \mathbb{N}$.
\end{lemma}

\begin{proof} First, put $\mathbb{R} = \bigcup_{n \in \mathbb{N}}[a_n,b_n]$ where $0 < b_n - a_n < \frac{1}{2}$ for every $n \in \mathbb{N}$. Since $|\mathbb{R}\setminus C| = \mathfrak{c}$, then there is $m \in \mathbb{N}$ such that $|[a_m,b_m] \cap (\mathbb{R}\setminus C)| = \mathfrak{c}$. Inside of this last set 
we can find $a_{n_1}, b_{n_1} \notin C$ such that $a_{n_1} < b_{n_1}$,   $|[a_{n_1},b_{n_1}] \cap (\mathbb{R}\setminus C)| = \mathfrak{c}$ and
$[a_{n_1}, b_{n_1}] \subseteq [a_m,b_m]$. Set  $A_1 := [a_{n_1},b_{n_1}] \cap (\mathbb{R}\setminus C)$. By continuing this procedure via  induction, for every $k \in \mathbb{N}$ we obtain a close interval $[a_{n_k},b_{n_k}]$ and $A_k$ such that 
\begin{enumerate}
\item[$\bullet$] $a_{n_k}, b_{n_k} \notin C$;

\item[$\bullet$] $0 < b_{n_k} - a_{n_k} < \frac{1}{2^k}$;

\item[$\bullet$] $A_k :=  [a_{n_k},b_{n_k}] \cap (\mathbb{R}\setminus C)$; and

\item[$\bullet$] $|A_k| = \mathfrak{c}$.
\end{enumerate}
This shows the lemma.
\end{proof}

\begin{theorem}\label{teoremaA}  If $C \subseteq  \mathbb{R}$ satisfies that  $|C| = |\mathbb{R}\setminus C| = \mathfrak{c}$, then there is $f \in Sel_2(\mathbb{R})$ such that $C \in \mathcal{M}_f \setminus \mathcal{B}_f(\mathbb{R})$.
\end{theorem}

\begin{proof}  We use the sets $\{ a_n : n \in \mathbb{N} \}$ and  $\{ b_n : n \in \mathbb{N} \}$ given by Lemma  \ref{lemaC} whose elements satisfy
that $a_{n}, b_{n} \notin C$, $b_n - a_n \longrightarrow 0$ and $0 < b_n - a_n $ for every $n \in \mathbb{N}$.
We know that there is $E \notin \mathcal{B}(\mathbb{R})$(besides, it is non-Lebesgue measurable) such that  $|E| = \mathfrak{c}$ and $E \subseteq (-1,2)$.
(see the book \cite[Th. 1.4.7, p. 32]{cohn}). Fix a bijection$\gamma: \mathbb{R} \to \mathbb{R} $ so that 
\begin{enumerate}
\item[$\bullet$] $\gamma(a_n)  = -1 - \frac{1}{n}$ and $\gamma(b_n)  = 2 + \frac{1}{n}$ for every $n \in \mathbb{N}$;

\item[$\bullet$] $\gamma(C) = E$; and

\item[$\bullet$] $\gamma(\mathbb{R} \setminus \big( C \cup \{ a_n : n \in \mathbb{N} \} \cup \{ b_n : n \in \mathbb{N} \} \big)  = \mathbb{R} \setminus \big(E \cup \{ -1 - \frac{1}{n} : n \in \mathbb{N} \} \cup \{ 2  + \frac{1}{n} : n \in \mathbb{N} \}\big)$.
\end{enumerate}
It follows from Lemma \ref{lemaB} that $C \notin   \gamma^{-1}\big(\mathcal{B}(\mathbb{R})\big) = \mathcal{B}_{f_\gamma}(\mathbb{R})$. Since 
$C \subseteq (a_n,b_n)_{f_\gamma}$ for every $n \in \mathbb{N}$, we must have that $\lambda_{f_\gamma}(C) = 0$, and hence we obtain that 
$C \in \mathcal{N}_{f_\gamma}  \subseteq \mathcal{M}_{f_\gamma} $.
\end{proof}

Next we generalize the technique used in the proof of the previous theorem. 

\medskip

Let $\gamma: \mathbb{R} \to \mathbb{R}$ be a bijection and $g \in Sel_2(\mathbb{R})$. Then, we define $f_{g,\gamma} \in Sel_2(\mathbb{R})$  by 
$r <_{f_{g,\gamma}} s$ if $\gamma(r) <_g \gamma(s)$ for each $\{r, s\} \in [\mathbb{R}]^2$. Observe  that if $g$ does not have neither minimal point nor maximal point, then $f_{g,\gamma}$ has the same property. We may also generalize Lemma \ref{lemaA}.

\begin{proposition}\label{lemaC}  If $\gamma: \mathbb{R} \to \mathbb{R}$ is a bijection and  $g \in Sel_2(\mathbb{R})$, then
\begin{enumerate}  
\item [$\bullet$] $(A,B)_{f_{g,\gamma}} = \gamma^{-1}\big((\gamma(A),\gamma(B))_g\big)$,

\item [$\bullet$] $[A,B)_{f_{g,\gamma}} = \gamma^{-1}\big([\gamma(A),\gamma(B)_g)\big)$,

\item [$\bullet$] $(A,B]_{f_{g,\gamma}} = \gamma^{-1}\big((\gamma(A),\gamma(B)]_g\big)$ and

\item [$\bullet$] $[A,B]_{f_{g,\gamma}} = \gamma^{-1}\big([\gamma(A),\gamma(B)]_g\big)$,
\end{enumerate}
for every  $A, B \in [\mathbb{R}]^{< \omega} \setminus \{\emptyset\}$ with $A \cap B = \emptyset$.
\end{proposition}

Besides, we may formulate the generalization of Theorem \ref{teoremaA}.

\begin{proposition}  If $\gamma: \mathbb{R} \to \mathbb{R}$ is a bijection and  $g \in Sel_2(\mathbb{R})$, then $\mathcal{B}_{f_{g,\gamma}}(\mathbb{R}) = \gamma^{-1}\big(\mathcal{B}_g(\mathbb{R})\big)$.\end{proposition}

Let $g$ be the two-point selection given in Example \ref{discrete} for which $\tau_g$ is the discrete topology. For this two-point selection, we have that 
$ \mathcal{B}_{f_{g,\gamma}}(\mathbb{R}) = \mathcal{P}(\mathbb{R})$ for every bijection $\gamma: \mathbb{R} \to \mathbb{R}$.
This fact suggests the following question:

\begin{question}\label{q3} Is there $g \in Sel_2(\mathbb{R})$ such that $|\{ \mathcal{B}_{f_{g,\gamma}(\mathbb{R})} \ | \ \gamma: \mathbb{R} \to \mathbb{R} \ \text{is a bijection} \ \}| = \mathfrak{c}$?
\end{question}

We could not answer this question, but it is possible to prove that this set is infinite for the Euclidian two-point selection $f_E$. Indeed, consider the quotient group $\mathbb{R}/\mathbb{Q}$ and  the non-Lebesgue measurable subsets  $ \{ E_n : n \in \mathbb{N}\}$ constructed in \cite[Th. 1.4.7, p. 32]{cohn} such that: 
\begin{enumerate}
\item[$\bullet$] $E_n \cap E_m = \emptyset$ for $n ,m \in \mathbb{N}$ with $n < m$;

\item[$\bullet$] If $E_n + \mathbb{Q} = \mathbb{R}$ for every $n \in \mathbb{N}$;

\item[$\bullet$] for every $n \in \mathbb{N}$, we have that $x - y \notin \mathbb{Q}$ for distinct points $x, y \in E_n$;  and

\item[$\bullet$] $(0,1) \subseteq \bigcup_{n \in \mathbb{N}}E_n \subseteq (-1,2)$. 
\end{enumerate}
Observe that $|E_n|= \mathfrak{c}$ for every $n \in \mathbb{N}$. For each positive integer $n \in \mathbb{N}$ we define a bijection $\gamma_n: \mathbb{R} \to \mathbb{R}$ so that:
\begin{enumerate}
\item[$\bullet$]  $\gamma_n$ is a bijection between $E_n$ and $(- \infty,-1)$;

\item[$\bullet$] $\gamma_n$  is the identity function on $E_m$ for all $m \in \mathbb{N} \setminus \{0, n\}$; and

\item[$\bullet$]  $\gamma_n$ is a bijection between $\mathbb{R} \setminus \big(\bigcup_{m \in \mathbb{N}}E_m\big)$ and $[-1,+\infty) \setminus  \big(\bigcup_{n \neq m \in \mathbb{N}}E_m\big)$.
\end{enumerate}
Then for all $n\in \mathbb{N}$ we have that $\gamma_n^{-1}((-\infty,-1)) = E_n \in \mathcal{B}_{f_{f_E,\gamma_n}}(\mathbb{R})$ and  $\gamma_n^{-1}(E_m) = E_m \notin \mathcal{B}_{f_{f_E,\gamma_n}}(\mathbb{R})$ for every $n \neq m \in \mathbb{N}$. Thus,  $\mathcal{B}_{f_{f_E,\gamma_n}}(\mathbb{R})  \neq \mathcal{B}_{f_{f_E,\gamma_m}}(\mathbb{R})$ for distinct $n ,m \in \mathbb{N}$.

\medskip

We finish with the formulation of a  question related to Theorem \ref{muchos}. 

\begin{question} In a model of $ZFC$,  are there $2^{2^{\mathfrak{c}}}$  many $\sigma$-algebras  on $\mathbb{R}$ which contain $[\mathbb{R}]^{\leq \omega}$?  
\end{question}

There is a general method of construction of $\sigma$-algebras on $\mathbb{R}$  which is nicely described in \cite[p. 100]{cie}:  If $\mathcal{I}$ is a $\sigma$-ideal and $\mathcal{A}$ is a $\sigma$-algebra,  both on $\mathbb{R}$,  then $\mathcal{A}[\mathcal{I}] := \{  A \triangle I  :  A \in \mathcal{A} \ \text{and} \ I \in \mathcal{I} \}$  is a
 $\sigma$-algebra on $\mathbb{R}$, where $A \triangle I = (A \setminus I) \cup (I \setminus A)$. Besides, if either $\mathcal{I}$ or   $\mathcal{A}$ contains 
 $[\mathbb{R}]^{\leq \omega}$, then  $\mathcal{A}[\mathcal{I}]$ is a  $\sigma$-algebra that contains $[\mathbb{R}]^{\leq \omega}$. In particular, if $\mathcal{N}$ is the 
 $\sigma$-ideal of Lebesgue-zero measure subsets of $\mathbb{R}$, then $\mathcal{B}(\mathbb{R})[\mathcal{N}]$ is the $\sigma$-algebra
 of Lebesgue-measurable subsets of $\mathbb{R}$, and if $\mathcal{M}$ is the $\sigma$-ideal of meager set of $\mathbb{R}$, then  $\mathcal{B}(\mathbb{R})[\mathcal{M}]$ is the $\sigma$-algebra consisting of the subsets of $\mathbb{R}$ with the Baire property. It seems natural to ask whether or not  these two $\sigma$-algebras are the Borel $\sigma$-algebras of topologies coming from two-point selections:
    
 \begin{question} Are there $f, g \in Sel_2(\mathbb{R})$  such that $\mathcal{B}_f(\mathbb{R})  = \mathcal{B}(\mathbb{R})[\mathcal{N}]$ and $\mathcal{B}_g(\mathbb{R})  = \mathcal{B}(\mathbb{R})[\mathcal{M}]$?  
\end{question}

 Finally, let us recall the definition of the Baire $\sigma$-algebra, denoted by $Baire(X)$, of a  Tychonoff space $X$: If $C(X)$ is the set of all continuous real-valued functions, then  $Baire(X)$ is the smallest $\sigma$-algebra on $X$ that makes each function in $C(X)$ measurable. Alternatively,  $Baire(X)$ is the $\sigma$-algebra generated by all sets of the form $\{ x \in X : F(x) > 0\}$ where $F \in C(X)$. Observe that $Baire(X) \subseteq  \mathcal{B}(X)$. Also we have that if $X$ is perfectly normal, then $Baire(X) = \mathcal{B}(X)$.  In particular, we obtain that $Baire(\mathbb{R}) = \mathcal{B}(\mathbb{R})$ and if $X$ is first countable, then $[X]^{\leq \omega} \subseteq Baire(X)$.  It would be nice to find $f \in Sel_2(\mathbb{R})$ such that $Baire((\mathbb{R},\tau_f))  \varsubsetneqq  \mathcal{B}_f(\mathbb{R})$.
Given $f \in Sel_2(\mathbb{R})$,  we may ask if $\mathfrak{B}$ is a base for the topology $\tau_f$, is there $g \in Sel_2(\mathbb{R})$ such that $\langle \mathfrak{B} \rangle = \mathcal{B}_g(\mathbb{R})$? Unfortunately, the answer to this question is negative.  Indeed, if $f \in Sel_2(\mathbb{R})$ is the two-point selection of Example \ref{discrete} such that $\tau_f$ is the discrete topology and $\mathfrak{B} := \{ \{r\} : r \in \mathbb{R}\}$, then $\langle \mathfrak{B} \rangle = \mathcal{C}(\mathbb{R}) \neq \mathcal{B}_g(\mathbb{R})$ for any $g \in Sel_2(\mathbb{R})$ (see  Theorem  \ref{cocountablenoBorel}).


\def\polhk#1{\setbox0=\hbox{#1}{\ooalign{\hidewidth
  \lower1.5ex\hbox{`}\hidewidth\crcr\unhbox0}}}
  \def\polhk#1{\setbox0=\hbox{#1}{\ooalign{\hidewidth
  \lower1.5ex\hbox{`}\hidewidth\crcr\unhbox0}}} \def\cprime{$'$}
  \def\cprime{$'$}

\bibliographystyle{plain}

\bibliography{bibbohr,bibpseudos}

\end{document}